\documentclass[a4paper,12pt]{amsart}
\usepackage{amsmath}
 \usepackage{latexsym, amsthm, amscd, euscript, float, subfig}
\usepackage{amssymb}
\usepackage{ifthen}
\usepackage{graphicx}
\usepackage{float}
\usepackage{cite}
\usepackage{amsfonts}
\usepackage{amscd}
\usepackage{amsxtra}
\usepackage{color}

\newtheorem{theorem}{Theorem}[section]
\newtheorem{lemma}[theorem]{Lemma}
\newtheorem{corollary}[theorem]{Corollary}

\theoremstyle{definition}

\numberwithin{equation}{section}

\def\be{\begin{equation}}
\def\ee{\end{equation}}

\newcounter{alphabet}


\begin{document}
\bibliographystyle{amsplain}
\title[Second Hankel determinant of logarithmic coefficients]{Second Hankel determinant of logarithmic coefficients of certain analytic functions}
\author[Vasudevarao Allu]{Vasudevarao Allu}
\address{Vasudevarao Allu\\School of Basic Science \\ Indian Institute of Technology Bhubaneswar\\
India}
\email{avrao@iitbbs.ac.in}

\author[Vibhuti Arora]{Vibhuti Arora}
\address{Vibhuti Arora\\School of Basic Science \\ Indian Institute of Technology Bhubaneswar\\
	India}
\email{vibhutiarora1991@gmail.com}


\begin{abstract}
We consider a
family of all analytic and univalent functions (i.e., one-to-one) in the unit disk $\mathbb{D}:=\{z\in \mathbb{C}:|z|<1\}$ of the form $f(z)=z+a_2z^2+a_3z^3+\cdots$. In this paper, we obtain the sharp bounds of the second Hankel determinant of Logarithmic coefficients for some subclasses of analytic functions.
\end{abstract}

\maketitle

\section{Introduction}\label{Introduction}

Let $\mathcal{A}$ denote the class of functions $f$ of the form
\begin{equation}\label{S}
f(z)=z+\sum_{n=2}^{\infty}a_nz^n,
\end{equation}
 which are analytic in the unit disk $\mathbb{D}:=\{z\in \mathbb{C}:|z|<1\}$. Let $\mathcal{S}$ be the class all functions $f\in \mathcal{A}$ that are univalent (i.e., one-to-one) in $\mathbb{D}$. For a general theory of univalent functions, we refer the classical books \cite{Dur83, Goo83}.
 
 For $q,n \in \mathbb{N}$, the {\em Hankel determinant} $H_{q,n}(f)$ of a function $f\in \mathcal{A}$ of the form \eqref{S} is defined as
 \begin{equation*}
 	H_{q,n}(f) := \begin{vmatrix}
 		a_n     & a_{n+1}  & \dots & a_{n+q-1} \\ 
 		a_{n+1}       &a_{n+2} & \dots & a_{n+q} \\
 		\vdots  & \vdots&  \vdots       & \vdots    \\
 		a_{n+q-1}      & a_{n+q} & \dots &a_{n+2(q-1)} \\ 
 	\end{vmatrix}.
 \end{equation*}
 General results for Hankel determinants of any degree with their applications can be
 found in \cite{Pom66,Pom67}. The problem of computing the bounds of Hankel determinants in a given
 family of analytic functions attracted the attention of many mathematicians (see \cite{ALT} and reference therein). 
 In particular, for $q=2$ and $n=1$, $H_{2,1}(f)=a_1a_3-a_2^2$
 is usually called the second Hankel determinant. For the class $\mathcal{S}$, the bound of $H_{2,1}(f)=a_1 a_3-a_2^2$ was estimated by Bieberbach in 1916.
 
 The {\em Logarithmic coefficients} $\gamma_n$ of $f\in \mathcal{S}$ are defined by the following series expansion:
 \begin{equation}\label{logcoefficients}
F_f(z):= \log \cfrac{f(z)}{z}=2\sum_{n=1}^{\infty}\gamma_n(f)z^n,\quad z\in \mathbb{D}.
 \end{equation}
The logarithmic coefficients have great importance as they play a crucial role in Milin conjecture \cite{Milin71} (see also \cite[p. 155]{Dur83}). Milin conjectured that for $f\in \mathcal{S}$ and $n\ge 2$,
$$
\sum_{m=1}^n\sum_{k=1}^{m}\bigg(k|\gamma_k|^2-\cfrac{1}{k}\bigg)\le 0,
$$
where the equality holds if, and only if, $f$ is a rotation of the Koebe function. De Branges \cite{DeB1} proved that Milin conjecture which confirmed the famous Bieberbach conjecture. On the other hand, one of reasons for more attention has been given to the Logarithmic coefficients is that the sharp bound for the class $\mathcal{S}$ is known only for $\gamma_1$ and $\gamma_2$, namely
\begin{equation}\label{gamma2}
|\gamma_1|\le 1 \text{ and } |\gamma_2|\le \cfrac{1}{2}(1+2e^{-2})=0.635\dots.
\end{equation}
It is still an open problem to find the sharp bounds of $\gamma_n,~n\ge 3$, for the class $\mathcal{S}$. Note that for the Koebe function $k(z)=z/(1-z)^2,\, z\in \mathbb{D}$, it is easy to see that $\gamma_n=1/n$ for each $n\ge 1$. Therefore it is expected that $|\gamma_n|\le 1/n$, since the Koebe function plays a role of extremal function in many problems of geometric function theory. But it was shown that, this is not true even for $n=2$, as we can seen in equation \eqref{gamma2}. The problem of finding the sharp bound of $|\gamma_n|$ for the class $\mathcal{S}$ and for its various subclasses are studied recently by several authors in different contexts, for instance see \cite{AV18, VAP18, Girela2000,AVA18, Pawel21}. 
 
 If $f$ is given by \eqref{S}, then by differentiating \eqref{logcoefficients} and equating coefficients, we obtain
 $$
 \gamma_1=\cfrac{1}{2}~a_2, \gamma_2=\cfrac{1}{2}\big(a_3-\cfrac{1}{2}~a_2^2\big), \text{ and } \gamma_3=\cfrac{1}{2}\big(a_4-a_2a_3+\cfrac{1}{3}~a_2^3\big).
 $$

Due to the great importance of logarithmic coefficients in the recent years, it is appropriate and interesting to compute the Hankel determinant whose entries are logarithmic coefficients. In particular, the second Hankel determinant of $F_f/2$ is defined as
\begin{equation}\label{Hankel}
H_{2,1}(F_f/2)=\gamma_1\gamma_3-\gamma_2^2=\cfrac{1}{4}\bigg(a_2a_4-a_3^2+\cfrac{1}{12}\,a_2^4\bigg).
\end{equation}
As usual, instead of the whole class $\mathcal{S}$ one can take into account their subclasses for which the problem of finding sharp estimates of Hankel determinant of logarithmic coefficients can be studied. The problem of computing the sharp bounds of $H_{2,1}(F_f/2)$ was considered in \cite{adam} for starlike and convex functions.

It is now appropriate to remark that $H_{2,1}(F_f/2)$ is invariant under rotation since for $f_{\theta}(z):=e^{-i\theta}f(e^{i\theta}z),\,\theta\in \mathbb{R}$ when $f\in \mathcal{S}$ we have
\begin{equation*}
H_{2,1}(F_{f_\theta}/2)=\cfrac{e^{4i\theta}}{4}\bigg(a_2a_4-a_3^2+\cfrac{1}{12}a_2^4\bigg)=e^{4i\theta}H_{2,1}(F_f/2).
\end{equation*}

The main purpose of this paper is to obtain the sharp upper
bounds of the second Hankel determinant of the logarithmic coefficients, {\it i.e.,} $|H_{2,1}(F_f/2)|$, for various subclasses of the class $\mathcal{A}$.

\section{Preliminary results}
In this section, we present key lemmas which will be used to prove the main results of this paper.
Let $\mathcal{P}$ denote the class of all analytic functions $p$ having positive real part in $\mathbb{D}$, with the form
\begin{equation}\label{P}
p(z)=1+c_1z+c_2z^2+c_3z^3\cdots.
\end{equation}
A member of $\mathcal{P}$ is called a {\em Carath\'{e}odory function}. It is known that $|c_n|\le 2,\, n\ge 1$ for a function $p\in \mathcal{P}$ (see \cite{Dur83}).\\[2mm]

Parametric representations of the coefficients are often useful. Libera and Z\l{}otkiewicz \cite{Lib82,Lib83} derived the following parameterizations of possible
values of $c_2$ and $c_3$. 
\begin{lemma}\cite{Lib82,Lib83}\label{Pclass}
If $p\in\mathcal{P}$ is of the form \eqref{P} with $c_1\ge 0$, then 
\begin{align}
c_1&=2p_1,\label{c2}\\
c_2&=2p_1^2+2(1-p_1^2)p_2,\label{c2c3}
\end{align}
and 
\begin{align}\label{c3}
c_3=2p_1^3+4(1-p_1^2)p_1 p_2-2(1-p_1^2)p_1p_2^2+2(1-p_1^2)(1-|p_2|^2)p_3
\end{align}
for some $p_1\in [0,1]$ and $p_2,\,p_3\in \overline {\mathbb{D}}:=\{z\in \mathbb{C}:|z|\le 1\}$.
	
For $p_1\in \mathbb{T}:=\{z\in \mathbb{C}:|z| =1\}$, there is a unique function $p\in \mathcal{P}$ with $c_1$ as in \eqref{c2}, namely 
$$
p(z)=\cfrac{1+p_1z}{1-p_1z}, \quad z\in \mathbb{D}.
$$
	
For $p_1\in \mathbb{D}$ and $p_2\in \mathbb{T}$, there is a unique function $p\in \mathcal{P}$ with $c_1$ and $c_2$ as in \eqref{c2} and \eqref{c2c3}, namely
\begin{equation}\label{P}
p(z)=\cfrac{1+(p_1+\overline{p_1}p_2)z+p_2z^2}{1-(p_1-\overline{p_1}p_2)z-p_2z^2}.
\end{equation}

For $p_1,~p_2\in \mathbb{D}$ and $p_3\in \mathbb{T}$, there is unique function $p\in \mathcal{P}$ with $c_1$, $c_2$, and $c_3$ as in  \eqref{c2}-\eqref{c3}, namely,
$$
p(z)=\cfrac{1+(\overline{p_2}p_3+\overline{p_1}p_2+p_1)z+(\overline{p_1}p_3+p_1\overline{p_2}p_3+p_2)z^2+p_3z^3}{1+(\overline{p_2}p_3+\overline{p_1}p_2-p_1)z+(\overline{p_1}p_3-p_1\overline{p_2}p_3-p_2)z^2-p_3z^3}\quad z\in \mathbb{D}.
$$
\end{lemma}

Next we recall the following well-known result due to Choi {\it et al.} \cite{CKS19}. Lemma \ref{Y(ABC)} plays an important role in the proof of our main results. 

\begin{lemma}\cite{CKS19}\label{Y(ABC)}
Let $A,\, B,\, C$ be real numbers and 
$$
Y(A,B,C):=\max_{z\in \overline{\mathbb{D}}}(|A+Bz+Cz^2|+1-|z|^2).
$$
(i) If $AC\ge0$, then 
$$
Y(A,B,C)=\left \{
\begin{array}{ll}
	|A|+|B|+|C|, & {\mbox{ for }} |B|\ge 2(1-|C|),
	\\[5mm]
	1+|A|+\cfrac{B^2}{4(1-|C|)}, & {\mbox{ for }} |B|< 2(1-|C|).
\end{array}
\right.
$$
(ii) If $AC<0$, then 
$$
Y(A,B,C)=\left \{
\begin{array}{ll}
		1-|A|+\cfrac{B^2}{4(1-|C|)}, & \quad -4AC(C^{-2}-1)\le B^2 \wedge |B|< 2(1-|C|),
	\\[5mm]
	1+|A|+\cfrac{B^2}{4(1+|C|)}, & \quad B^2<\min
	\{4(1+|C|)^2,-4AC(C^{-2}-1)\},
	\\[5mm]
	R(A,B,C), &\quad\text{ otherwise},
\end{array}
\right.
$$
where
$$
R(A,B,C)=\left \{
\begin{array}{ll}
		|A|+|B|+|C|, & \quad |C|(|B|+4|A|)\le |AB|,
	\\[5mm]
		-|A|+|B|+|C|, & \quad |AB|\le |C|(|B|-4|A|),
	\\[5mm]
	(|A|+|C|)\sqrt{1-\cfrac{B^2}{4AC}}\,, &\quad\text{ otherwise}.
\end{array}
\right.
$$
\end{lemma}

\section{Main results}
For a better clarity in our presentation, we divide this section into several subsections 
consisting of different families of functions from the class $\mathcal{A}$ and prove our main results
associated with those classes of functions. 

\subsection{The class $\mathcal{S}_{\beta}(\alpha)$}
\medskip\noindent

To state our first result we need to introduce the following definitions: A function $f\in \mathcal{A}$ is called {\em starlike} if $f(\mathbb{D})$ is a starlike domain with respect to origin. The
class of univalent starlike functions is denoted by $\mathcal{S}^*$. 
There is one natural generalization of starlike functions is $\beta$-spirallike functions of order $\alpha$ which leads to a useful criterion for univalency. The family $\mathcal{S}_{\beta}(\alpha)$ of 
{\em $\beta$-spirallike functions of order $\alpha$} is defined
by
$$
\mathcal{S}_{\beta}(\alpha)=\bigg\{ f\in \mathcal{A}:{\rm Re}\bigg(e^{-i\beta}\cfrac{zf'(z)}{f(z)}\bigg)>\alpha \cos \beta  \bigg\},
$$
where $0\le \alpha< 1$ and $-\pi/2< \beta< \pi/2$. It is known that each function in $\mathcal{S}_{\beta}(\alpha)$ is univalent in $\mathbb{D}$ (see \cite{Lib67}). Functions in $\mathcal{S}_\beta(0)$  are called \textit{$\beta$-spirallike}, but they do not necessarily
belong to the starlike family $\mathcal{S}^*$. For example, the function
$f(z)=z(1-iz)^{i-1}$ is ${\pi/4}$-spirallike but $f \notin\mathcal{S}^*$.
The class $\mathcal{S}_\beta (0)$ was introduced by ${\rm \check{S}}$pa${\rm\check{c}}$ek \cite{Spacek-33}
(see also \cite{Dur83}). Moreover, $\mathcal{S}_0 (\alpha)=:\mathcal{S}^*{(\alpha)}$ is the usual class of starlike functions of order $\alpha$,
and $\mathcal{S}^*(0)=\mathcal{S}^*$. Recall that the class  $\mathcal{S}_{\beta}(\alpha)$, for $0\le \alpha<1$, is studied by several authors in different perspective (see, for instance \cite{KS20,Lib67}).\\[2mm]

Now we will prove the first main result of this paper.
\begin{theorem} \label{a2a1}
	Let $-\pi/2< \beta< \pi/2$ and $0\le \alpha <1$. For every $f\in \mathcal{S}_{\beta}(\alpha)$ of the form \eqref{S}, we have
	\begin{equation}\label{inequalityS}
		|H_{2,1}(F_f/2)|\le \cfrac{(1-\alpha)^2\cos^2 \beta}{4}.
	\end{equation}
	Equality in \eqref{inequalityS} holds only for the rotation of the function
	$$
	f_1(z)=\cfrac{z}{(1-z^2)^{(1-\alpha)\cos \beta e^{i\beta}}}.
	$$
\end{theorem}

\begin{proof}
	Let $f(z)=z+\sum_{n=2}^{\infty}a_nz^n\in \mathcal{S}_{\beta}(\alpha)$. Then by the definition, we may consider  $p(z)=1+c_1z+c_2z^2+\cdots\in \mathcal{P}$ of the form 
	$$
	p(z)=\cfrac{1}{1-\alpha}\bigg\{ \cfrac{1}{\cos \beta}\bigg(e^{-i\beta} \cfrac{zf'(z)}{f(z)}+i\sin \beta\bigg)-\alpha\bigg\}.
	$$
	The above equality we can rewrite as 
$$
((1-\alpha)p(z)+\alpha)\cos \beta-i \sin \beta=e^{-i\beta}\cfrac{zf'(z)}{f(z)}.
$$
By using the Taylor series representations of the functions $f$ and $p$, and after comparing the coefficients of $z^n\, (n=1,2,3)$ on both the sides, we get
\begin{align*}
	a_2&=(1-\alpha)e^{i\beta}\cos\beta c_1,\\ \text
	2a_3&=(1-\alpha)^2e^{2i\beta}\cos^2\beta c_1^2+(1-\alpha)e^{i\beta}\cos\beta c_2,
\end{align*}
and
\begin{align*}
	6a_4&=(1-\alpha)^3e^{3i\beta}\cos^3\beta c_1^3+3c_1c_2(1-\alpha)^2e^{2i\beta}\cos^2\beta+2(1-\alpha)e^{i\beta}\cos\beta c_3.
\end{align*}
Substitution of $a_2$, $a_3$, and $a_4$ in \eqref{Hankel} gives
\begin{align*}
	H_{2,1}(F_f/2)=\cfrac{(1-\alpha)^2e^{2i\beta}\cos^2\beta}{48}~(4c_1c_3-3c_2^2).
\end{align*}
As $H_{2,1}(F_f/2)$ and $\mathcal{S}_{\beta}(\alpha)$ are invariant under the rotations,  therefore to simplify the calculation we assume that $c_1$ is real. Therefore, by Lemma \ref{Pclass}, for some $p_1\in [0,1]$ and $p_2,p_3\in \overline{\mathbb{D}}$ we have
\begin{align}\label{eqsba}
H_{2,1}(F_f/2)&=\cfrac{(1-\alpha)^2\cos^2 \beta}{12}\bigg(p_1^4+2(1-p_1^2)p_1^2p_2-(1-p_1^2)(3+p_1^2)p_2^2\\&\quad\quad+4p_1(1-p_1^2)(1-|p_2|^2)p_3\bigg)\nonumber.
\end{align}
Now, we may have the following cases on $p_1$:\\

\medskip\noindent
{\bf Case 1:} Let $p_1=1$. Then from \eqref{eqsba} we get
$$
|H_{2,1}(F_f/2)|=\cfrac{(1-\alpha)^2\cos^2 \beta}{12}.
$$

\medskip\noindent
{\bf Case 2:} Let $p_1=0$. Then from \eqref{eqsba} we get
$$
|H_{2,1}(F_f/2)|=\cfrac{(1-\alpha)^2\cos^2 \beta}{12}~|3p_2^2|\le \cfrac{(1-\alpha)^2\cos^2 \beta}{4}.
$$

\medskip\noindent
{\bf Case 3:} Let $p_1\in (0,1)$.
Applying the triangle inequality in \eqref{eqsba} and by using the fact that $|p_3|\le 1$, we obtain
\begin{align}\label{eq1}
	|H_{2,1}(F_f/2)|&\le\cfrac{(1-\alpha)^2\cos^2 \beta}{12}\bigg(\big|p_1^4+2(1-p_1^2)p_1^2p_2-(1-p_1^2)(3+p_1^2)p_2^2\big|\\&\quad\quad+4p_1(1-p_1^2)(1-|p_2|^2)\bigg)\nonumber\\
	&=\cfrac{(1-\alpha)^2\cos^2 \beta p_1(1-p_1^2)}{3}\bigg(\big|A+Bp_2+Cp_2^2\big|+1-|p_2|^2\bigg),
\end{align}
where
$$
A:=\cfrac{p_1^3}{4(1-p_1^2)},~B:=\cfrac{p_1}{2}, \text{ and } C:=-\cfrac{(3+p_1^2)}{4p_1}.
$$
Observe that $AC<0$, so we can apply case (ii) of Lemma \ref{Y(ABC)} and we obtain that the only condition $|AB|-|C|(|B|-4|A|)\le 0$ is satisfying for some $p_1$, as we can see from below observations.

{\bf 3(a)} Note that for $p_1\in (0,1)$ we have
$$
-4AC\biggl(\cfrac{1}{C^2}-1\biggr)-B^2=-\cfrac{3p_1^2}{p_1^2+3}\le 0.
$$
Also, the inequality $|B|<2(1-|C|)$ is equivalent to $2p_1^2-4p_1+3<0$ which is not true for $p_1\in (0,1)$. 

{\bf 3(b)} Next, it is easy to check that
$$
\min \biggl\{4(1+|C|)^2,-4AC\bigg(\cfrac{1}{C^2}-1\bigg)\biggr\}=-4AC\bigg(\cfrac{1}{C^2}-1\bigg)\le B^2,
$$
here the last inequality directly follows from 3(a).

{\bf 3(c)} For $0< p_1< 1$, it is easy to verify that $|C|(|B|+4|A|)-|AB|\le 0$ is not satisfied as $3+4p_1^2\ge 0$.

{\bf 3(d)} We note that the inequality
$$
|AB|-|C|(|B|-4|A|)=\cfrac{4p_1^4+8p_1^2-3}{8(1-p_1^2)}\le 0
$$
holds for $0<p_1\le s_1:=\sqrt{\sqrt{7}/2-1}\approx 0.568221$. 
It follows from Lemma \ref{Y(ABC)} and the inequality \eqref{eq1} that
\begin{align*}
|H_{2,1}(F_f/2)|&\le \cfrac{(1-\alpha)^2\cos^2 \beta p_1(1-p_1^2)}{3}(-|A|+|B|+|C|)\\
&=\cfrac{(1-\alpha)^2\cos^2 \beta(3-4p_1^4)}{12}\\
&\le \cfrac{(1-\alpha)^2\cos^2 \beta}{4}
\end{align*}
for $0<p_1\le s_1$.

{\bf 3(e)} For $s_1<p_1<1$, we use the last case of Lemma \ref{Y(ABC)} together with \eqref{eq1} to obtain
\begin{align*}
|H_{2,1}(F_f/2)|&\le \cfrac{(1-\alpha)^2\cos^2 \beta p_1(1-p_1^2)}{3}~(|C|+|A|)\sqrt{1-\cfrac{B^2}{4AC}}=t(p_1),
\end{align*}
where
$$
t(x):=\cfrac{(1-\alpha)^2\cos^2 \beta(3-2x^2)}{6\sqrt{3+x^2}}.
$$
Observe that 
$$
t'(x)=-\cfrac{(1-\alpha)^2\cos^2 \beta(15x+2x^3)}{6(3+x^2)^{3/2}}<0, \quad s_1<x<1,
$$
Thus, the function $t$ is decreasing on $s_1<x<1$ which yields 
$$
t(x)\le t(s_1)=(1-\alpha)^2\cos^2 \beta \cfrac{ 5-\sqrt{7}}{3\sqrt{8+2\sqrt{7}}}\le \cfrac{(1-\alpha)^2\cos^2 \beta}{4},
$$
for $s_1<x<1$.
Summarizing parts from Case 1-3, it follows the desired inequality \eqref{inequalityS}. \\[2mm]

We now proceed to prove the equality part. Consider the function
$$
f_1(z)=\cfrac{z}{(1-z^2)^{(1-\alpha)\cos \beta e^{i\beta}}}, \quad {z\in \mathbb{D}}.
$$
A simple calculation shows that $f_1$ belongs to $\mathcal{S}_{\beta}(\alpha)$. The coefficients of $f_1$ are $a_2=0$ and $a_3=(1-\alpha)\cos \beta e^{i\beta}$. Then from \eqref{Hankel} we see that the inequality \eqref{inequalityS} is sharp for $f_1$. 
This completes the proof.
\end{proof}

For the special case $\beta=0$, we get the following sharp result for the class of starlike functions of order $\alpha$:
\begin{corollary}
	Let $f \in \mathcal{S}^*(\alpha)$, $0 \le \alpha <1$. Then we have 
		$$
	|H_{2,1}(F_f/2)|\le \cfrac{(1-\alpha)^2}{4}.
	$$
	The inequality is sharp.
\end{corollary}

For $\alpha=0$ and $\beta=0$, we obtain the estimate for the class $\mathcal{S}^*$ of starlike function.
\begin{corollary}
Let $f \in \mathcal{S}^*$. Then we have 
$$
|H_{2,1}(F_f/2)|\le \cfrac{1}{4}.
$$
The equality holds for the rotation of the Koebe function.
\end{corollary}

\subsection{The class ${\mathcal G}(\gamma)$}
\medskip\noindent

Recall that a function $f \in \mathcal{A}$ is said to be {\em locally univalent function} at a point $z\in \mathbb{D}$ if it is univalent in some neighborhood of $z$; equivalently $f'(z)\ne 0$.
Let $\mathcal{LU}$ denote the subclass of $\mathcal{A}$ consisting of all locally univalent functions;
namely, $\mathcal{LU}=\{f\in\mathcal{A}:f'(z)\ne0,z\in\mathbb{D}\}$. A family $\mathcal{G}(\nu)$, $\nu>0$, of functions $f\in \mathcal{LU}$ is defined by
$$
{\mathcal G}(\nu)=\left\{f\in \mathcal{LU}:{\rm Re } \left ( 1+\frac{zf''(z)}{f'(z)}\right )<1+\cfrac{\nu}{2}\right\}.
$$
The class $\mathcal{G}:=\mathcal{G}(1)$ was first introduced by Ozaki \cite{Ozaki41} and proved the inclusion relation $\mathcal{G}\subset \mathcal{S}$.
The Taylor coefficient problem for the class $\mathcal{G}(\nu)$, $0<\nu\leq 1$, is discussed in \cite{Obradovic13}. Recently, the radius of convexity for functions in the class $\mathcal{G}(\nu)$, $\nu>0$, is obtained in \cite{Shankey20}. The class $\mathcal{G}(\nu)$, with special choices of the parameter $\nu$, has also been considered by many researchers in the literature for different purposes; see for instance \cite{ponnusamy96,ponnusamy07,Shankey20}. \\[2mm]

Next, we obtained the following sharp bound of  $|H_{2,1}(F_f/2)|$ for $f\in \mathcal{G}(\nu)$.
\begin{theorem}\label{glambda}
	Let $0<\nu\le 1$. If $f\in \mathcal{G}(\nu)$ given by \eqref{S}, then
	$$
	|H_{2,1}(F_f/2)|\le \cfrac{\nu^2(\nu^2+12\nu-44)}{192(\nu^2+8\nu-32)}.
	$$
	The inequality is sharp. 
\end{theorem}
\begin{proof}
	Since $f\in \mathcal{G}(\nu)$, then there exist a Carath\'{e}odory function $p$ of the form 
	$$
	p(z)=\cfrac{1}{\nu} \bigg(\nu- \cfrac{2zf''(z)}{f'(z)}\bigg).
	$$ 
	It is equivalent to write 
	\begin{equation}\label{eqg}
		\nu(p(z)-1)f'(z)=-2zf''(z).
	\end{equation}
	By using the Taylor series representations for functions $f$ and $p$ and equating the coefficients of $z$, $z^2$, and $z^3$ in \eqref{eqg}, we obtain
	\begin{equation*}
		a_2=\cfrac{\nu c_1}{4}, a_3=\cfrac{\nu(\nu c_1^2-2c_2)}{24}~, \text{ and } a_4=\cfrac{\nu(6\nu c_1c_2-8c_3-\nu^2 c_1^2)}{192}.
	\end{equation*}
	By substituting the above expression for $a_2,~a_3$, and $a_4$ in \eqref{Hankel} and then further simplification gives
	\begin{align*}
		H_{2,1}(F_f/2)&=\cfrac{1}{4}\bigg(a_2a_4-a_3^2+\cfrac{1}{12}\,a_2^4\bigg)\\
		&=\cfrac{\nu^2}{36864}~\big(96c_1c_3-64c_2^2-8\nu c_1^2 c_2-\nu^2 c_1^4\big).
	\end{align*}
Noting that $\mathcal{G}(\nu)$ and $H_{2,1}(F_f/2)$ are rotationally invariant. So we can assume that $c_1$ is real. Thus, in view of Lemma \ref{Pclass} and writing $c_1$, $c_2$, and $c_3$ in terms of $p_1$, $p_2$, and $p_3$ we obtain
	\begin{align}\label{hankelg}
		H_{2,1}(F_f/2)&=\cfrac{\nu^2}{2304}\bigg((-\nu^2-4\nu+8)p_1^4+4(4-\nu)(1-p_1^2)p_1^2p_2\\
		&\quad\quad\quad-8(2+p_1^2)(1-p_1^2)p_2^2+24(1-p_1^2)(1-|p_2|^2)p_1p_3\bigg)\nonumber
	\end{align}
	with $p_1\in [0,1]$ and $p_2,p_3\in \overline{\mathbb{D}}$.
	
		We next divide the proof into three cases:
		
	\medskip\noindent
	{\bf Case 1:} If $p_1=1$. Then from \eqref{hankelg}, we obtain
	$$
	|H_{2,1}(F_f/2)|=\cfrac{\nu^2(-\nu^2-4\nu+8)}{2304}.
	$$
	
	\medskip\noindent
	{\bf Case 2:} If $p_1=0$. Then from \eqref{hankelg}, we obtain
	$$
	|H_{2,1}(F_f/2)|=\cfrac{16\nu^2|p_2|^2}{2304}\le \cfrac{\nu^2}{144}.
	$$
	
	\medskip\noindent
	{\bf Case 3:} Now let $p_1\in (0,1)$. Then use $|p_3|\le 1$ in \eqref{hankelg} to obtain
	\begin{align*}
		&|H_{2,1}(F_f/2)|\\&\le \cfrac{24\nu^2p_1(1-p_1^2)}{2304}\bigg(\bigg|\cfrac{p_1^3(-\nu^2-4\nu+8)}{24(1-p_1^2)}+\cfrac{(4-\nu)p_1p_2}{6}-\cfrac{(2+p_1^2)p_2^2}{3p_1}\bigg|+1-|p_2|^2\bigg)\\
		&=\cfrac{24\nu^2p_1(1-p_1^2)}{2304}\big(|A+Bp_2+Cp_2^2|+1-|p_2|^2\big),
	\end{align*}
	where
	\begin{equation*}
		A:=\cfrac{p_1^3(-\nu^2-4\nu+8)}{24(1-p_1^2)},~B:=\cfrac{(4-\nu)p_1}{6}, \text{ and }C:=-\cfrac{(2+p_1^2)}{3p_1}.
	\end{equation*}
	Since $AC<0$, then we apply Lemma \ref{Y(ABC)} only for the case (ii) and we obtain that only $|AB|-|C|(|B|-4|A|)\le 0$ is satisfying for some $p_1$, as we can see from below observations.
	
	{\bf 3(a)} Note that
	$$
	-4AC\biggl(\cfrac{1}{C^2}-1\biggr)-B^2=\cfrac{p_1^2\big(-\nu^2 p_1^2+(2\nu^2+16\nu-32)\big)}{12(p_1^2+2)}\le 0,
	$$
	for $0<p_1<1$ and $0<\nu \le 1$. Moreover, it is easy to see that the inequality $|B|<2(1-|C|)$ does not hold for $p_1\in (0,1)$. 
	
	{\bf3(b)} Using the above observation, we have the following inequality 
	$$
	\min \biggl\{4(1+|C|)^2,-4AC\bigg(\cfrac{1}{C^2}-1\bigg)\biggr\}=-4AC\bigg(\cfrac{1}{C^2}-1\bigg)\le B^2.
	$$
	
	{\bf3(c)} We now show that $|C|(|B|+4|A|)-|AB|> 0$ holds for all $\nu\in(0,1]$ and $p_1\in (0,1)$. A simple calculation shows that
	$$
	|C|(|B|+4|A|)-|AB|=\cfrac{-p_1^4(8+\nu)\nu^2+8p_1^2(-\nu^2-4\nu+8)+16(4-\nu)}{144(1-p_1^2)}:=g(\nu).
	$$
	It is easily check that $g$ is a decreasing function with respect to $\nu$ in $(0,1]$. This implies that
	$$
	g(\nu)\ge g(1)=\cfrac{16+8p_1^2-3p_1^4}{48(1-p_1^2)}\ge 0.
	$$

	{\bf 3(d)} Next, the inequality
	\begin{align*}
		|AB|- |C|(|B|-4|A|)&=\cfrac{p_1^4(\nu^3-8\nu^2-64\nu+128)+8p_1^2(-2\nu^2-9\nu+20)-16(4-\nu)}{144(1-p_1^2)}\\&\le 0
	\end{align*}
	is equivalent to
	$$
	G(x^2)\le 0, \quad \nu\in(0,1] \text{ and } x\in (0,1)
	$$
	where
	$$
	G(x):=x^2(\nu^3-8\nu^2-64\nu+128)+8x(-2\nu^2-9\nu+20)-16(4-\nu) \text{ and } x=p_1^2.
	$$ 
	which is a quadratic polynomial. Note that the discriminant of $G$ is given by $\Delta=192 (304 - 248 \nu + 11 \nu^2 + 16 \nu^3 +\nu^4)>0$ for $\nu
	\in(0,1]$. The equation $G(x)=0$ has following two solutions:
	$$
	x_{1}:=\cfrac{-4(-2\nu^2-9\nu+20)-4 \sqrt{3(304-248\nu+11\nu^2+16\nu^3+\nu^4)}}{\nu^3-8\nu^2-64\nu+128}
	$$
	and
	$$
	x_{2}:=\cfrac{-4(-2\nu^2-9\nu+20)+ 4 \sqrt{3(304-248\nu+11\nu^2+16\nu^3+\nu^4)}}{\nu^3-8\nu^2-64\nu+128}.
	$$
	Check that $x_1<0$ and $x_2>0$ as $892 - 735\nu + 35\nu^2 + 48 \nu^3 + 3 \nu^4>0$. Also it is easy to verify that
	$x_2<1$ since $-28672 + 29696 \nu - 2816\nu^2 - 2848\nu^3 - 8 \nu^4 + 32\nu^5 - \nu^6<1$. Therefore, the function $G(x)$ has the unique zero $x_2\in (0,1)$.
	Hence $G\le0$ for $0<x\le x_2$  and the condition $|AB|\le |C|(|B|-4|A|)$ in Lemma \ref{Y(ABC)} is satisfied for $0<p_1\le \sqrt{x_2}$. Therefore, Lemma \ref{Y(ABC)} yields
	\begin{align}\label{pjieq}
		|H_{2,1}(F_f/2)|&\le \cfrac{24\nu^2p_1(1-p_1^2)}{2304}(-|A|+|B|+|C|)\nonumber \\&=\cfrac{\nu^2}{2304}~\bigg( (\nu^2+8\nu-32)p_1^4+(8-4\nu)p_1^2+16\bigg)=:\phi(p_1)\\
		&\le \phi(s_2)=\cfrac{\nu^2(\nu^2+12\nu-44)}{192(\nu^2+8\nu-32)}\nonumber
	\end{align}
	where 
	\begin{equation}\label{p0}
		s_2:=\sqrt{\cfrac{2(\nu-2)}{\nu^2+8\nu-32}}
	\end{equation}
	is the critical point of $\phi$ and gives the maximum value. A more involved computation shows that $s_2<\sqrt{x_2}$.
	
	{\bf 3(e)} Next consider the case $\sqrt{x_2}\le p_1<1$ and use the last case of the Lemma \ref{Y(ABC)}
	\begin{align}\label{eqbound}
		|H_{2,1}(F_f/2)|&\le \cfrac{24\nu^2p_1(1-p_1^2)}{2304}	(|A|+|C|)\sqrt{1-\cfrac{B^2}{4AC}}\nonumber \\&=\cfrac{\nu^2}{2304}~\bigg( 16-8p_1^2-p_1^4\nu(\nu+4)\bigg)\sqrt{\cfrac{3(p_1^2\nu^2+\nu^2+8\nu-16)}{2(2+p_1^2)(\nu^2+4\nu-8)}}=:\psi(p_1).
	\end{align}
Since
	\begin{align*}
		\psi'(x)&=\cfrac{-\nu^2x}{2304(2+x^2)^2(\nu^2+4\nu-8)}~\sqrt{\cfrac{3(2+x^2)(\nu^2+4\nu-8)}{2(x^2\nu^2+\nu^2+8\nu-16)}}\times\bigg( 16  (-48 + 24 \nu + \nu^2) \\&\quad+	8 x^2 (-16 - 56 \nu + 23 \nu^2 + 12 \nu^3 + \nu^4)+ 
		x^4 \nu (-192 + 64 \nu + 76 \nu^2 + 13 \nu^3) + 
	\\
	&\quad+4 x^6 \nu^3 (4 + \nu)\bigg)\\ 
		&<0,
	\end{align*}
	so  $\psi$ is decreasing in the interval $[\sqrt{x_2}, 1)$. Therefore $\psi(p_1)\le \psi(\sqrt{x_2})$ and equation \eqref{eqbound} leads to
	\begin{align*}
		|H_{2,1}(F_f/2)|&\le \psi(\sqrt{x_2})=\cfrac{-16\sqrt{3}\nu^2(\nu^2+4\nu-8)(a+b\sqrt{c})}{2304d^2}\sqrt{\cfrac{e+4\nu^2\sqrt{c}}{f+2(\nu^2+4\nu-8)\sqrt{c}}},
	\end{align*}
	where
	\begin{align*}
		a&:=3\nu^4+54\nu^3+18\nu^2-984\nu+1344\\
		b&:=2\nu^2+10\nu-16\\
		c&:=3(304-248\nu+11\nu^2+16\nu^3+\nu^4)\\
		d&:=128-64\nu-8\nu^2+\nu^3\\
		e&:=-2048+2048\nu-336\nu^2-108\nu^3+8\nu^4+\nu^5\\
		f&:=(\nu^3-4\nu^2-46\nu+88)(\nu^2+4\nu-8).
	\end{align*}
	A lengthy calculation shows that 
	$$
	\psi(\sqrt{x_2})=\phi(\sqrt{x_2})
	$$
	and by using \eqref{pjieq} we deduce that
	$$
	|H_{2,1}(F_f/2)|\le \phi(\sqrt{x_2})\le \phi(s_2)=\cfrac{\nu^2(\nu^2+12\nu-44)}{192(\nu^2+8\nu-32)}.
	$$
	Thus combining all the above cases 1-3, we find the desired inequality.\\[2mm]
	
	To prove the equality part, consider the function
	$$
	p_2(z)=\cfrac{1-z^2}{1-2s_2z+z^2} 
	$$
	is in the class $\mathcal{P}$ follows from Lemma \ref{P}. Here $s_2$ is defined by \eqref{p0}. For given $p_2\in \mathcal{P}$, we recall from \eqref{eqg} that the function $f_{2}\in \mathcal{G}(\nu)$
	with 
	$$
	a_2=-\cfrac{\nu s_2}{2},~a_3=\cfrac{\nu(1+(\nu-2)s_2^2)}{6}, \text{ and } a_4=-\cfrac{\nu(\nu-2)s_2(3+(\nu-4)s_2^2)}{24}.
	$$
	Hence 
	$$
	|H_{2,1}(F_f/2)|= \cfrac{\nu^2(\nu^2+12\nu-44)}{192(\nu^2+8\nu-32)}.
	$$
	This completes the proof of Theorem \ref{glambda}.
\end{proof}

\subsection{The class $\mathcal{F}_0 (\lambda)$}
\medskip\noindent

Let $f\in\mathcal{A}$ be a locally univalent. Then, according to Kaplan's theorem, it follows that
$f$ is {\em close-to-convex} if, and only if,
$$
\int_{\theta_1}^{\theta_2} {\rm Re}\left(1+\frac{zf''(z)}{f'(z)}\right)\,d\theta>-\pi,\quad z=re^{i\theta},
$$
for each $r~(0<r<1)$
and for each pair of real numbers $\theta_1$ and $\theta_2$ with $\theta_1<\theta_2$. If a locally univalent analytic function $f$ defined in $\mathbb{D}$ satisfies
$$
{\rm Re}\left(1+\frac{zf''(z)}{f'(z)}\right)>-\frac{1}{2},  ~\mbox{ for $z \in \mathbb{D}$},
$$
then by the Kaplan characterization it follows easily that $f$ is close-to-convex in $\mathbb{D}$,
and hence $f$ is univalent in $\mathbb{D}$.
This generates the following subclass of the class of close-to-convex (univalent) functions:
$$
\mathcal{C} (-1/2):=\left\{f\in \mathcal{A}:\,{\rm Re}\left (1+ \frac{zf''(z)}{f'(z)}\right )>-\frac{1}{2}  ~\mbox{ for $z \in \mathbb{D}$}\right\}.
$$
Functions in $\mathcal{C} (-1/2)$ are not necessarily starlike but is convex in some
direction. Other related results for $f\in \mathcal{C} (-1/2)$ were also obtained in \cite{AS17, PSY14}.
Robertson \cite{Robert36} considered the following generalization of $\mathcal{C} (-1/2)$ for $-1/2<\lambda\le 1/2$. The class $\mathcal{F}(\lambda)$, defined for $-1/2< \lambda\le 1$ by 
\begin{equation*}
\mathcal{F}(\lambda)=\left\{f\in \mathcal{A}:\,	{\rm Re}\left(1+\frac{zf''(z)}{f'(z)}\right)>\frac{1}{2} -\lambda ~\mbox{ for $z \in \mathbb{D}$}\right\}.
\end{equation*}
 We note that $\mathcal{F}(1)=\mathcal{C}(-1/2)$. Moreover, $\mathcal{F}(1/2)=:\mathcal{C}$ is the usual class of convex functions. Functions in $\mathcal{F}(\lambda)$  are close-to-convex for $1/2\le \lambda\le 1$ but contain non-starlike functions for all $1/2< \lambda\le 1$ (see \cite{PRU76}). 
The class $\mathcal{F}(\lambda)$ has also been considered for the restriction $1/2\le \lambda\le 1$, denote by $\mathcal{F}_0 (\lambda)$, and further extensively studied in this regards, we refer to \cite{ALT}. 

In the next theorem, we will discuss about the sharp bound for $|H_{2,1}(F_f/2)|$ when the functions $f$ runs over the class $ \mathcal{F}_0 (\lambda)$.

\begin{theorem}
Let $f\in \mathcal{F}_0 (\lambda)$, for $1/2\le \lambda\le 1$, given by \eqref{S}. Then 
$$
|H_{2,1}(F_f/2)|\le \cfrac{(2\lambda+1)^2(12\lambda^2-60\lambda-165)}{576(4\lambda^2-12\lambda-39)}.
$$	
The inequality is sharp. 
\end{theorem}
\begin{proof}
Let $f\in \mathcal{F}_0 (\lambda)$ be of the form \eqref{S}. Then there exists $p\in \mathcal{P}$ of the form \eqref{P} such that 
\begin{equation}\label{c(-1/2)p}
p(z)=\cfrac{2}{2\lambda+1}\bigg(\cfrac{zf''(z)}{f'(z)}+\cfrac{2\lambda+1}{2}\bigg), \quad z\in \mathbb{D}.
\end{equation}
Substituting the series \eqref{S} and \eqref{Pclass} into \eqref{c(-1/2)p} and equating the coefficients we obtain
\begin{align*}
a_2&=\cfrac{2\lambda+1}{4}\,c_1,\\
a_3&=\cfrac{2\lambda+1}{24}(2c_2+(2\lambda+1)c_1^2),
\end{align*}
and
\begin{align*}
 a_4&=\cfrac{(2\lambda+1)}{192}(8c_3+6(2\lambda+1)c_1c_2+c_1^3(2\lambda+1)^2).
\end{align*}
 Since the class $\mathcal{F}_0(\lambda)$ and the functional $H_{2,1}(F_f/2)$ are rotationally invariant, without loss of generality we may assume that $c_1\in [0,2]$. Hence by substituting $a_2$, $a_3$, and $a_4$ in \eqref{Hankel}, and Lemma \ref{Pclass}, we obtain
\begin{align}\label{hankelp}
H_{2,1}(F_f/2)&=\cfrac{1}{4}\bigg(a_2a_4-a_3^2+\cfrac{1}{12}\,a_2^4\bigg)\\&=\cfrac{(2\lambda+1)^2}{36864}\,\big(96c_1c_3+8(2\lambda+1)c_1^2c_2-(2\lambda+1)^2c_1^4-64c_2^2\big)\nonumber\\
&=\cfrac{(2\lambda+1)^2}{2304}\,\bigg((-4\lambda^2+4\lambda+11)p_1^4+4(2\lambda+5)(1-p_1^2)p_1^2p_2\nonumber\\
&\quad\quad-8(p_1^2+2)(1-p_1^2)p_2^2+24(1-p_1^2)(1-|p_2|^2)p_1 p_3\bigg).
\end{align}
The following three possibilities arise:\\

\medskip\noindent
{\bf Case 1:} If $p_1=1$. Then by \eqref{hankelp} we have
$$
|H_{2,1}(F_f/2)|=\cfrac{(2\lambda+1)^2}{2304}\,((-4\lambda^2+4\lambda+11)p_1^4).
$$

\medskip\noindent
{\bf Case 2:} If $p_1=0$. Then by \eqref{hankelp} we have
$$
|H_{2,1}(F_f/2)|=\cfrac{(2\lambda+1)^2}{144}~|p_2|^2\le \cfrac{(2\lambda+1)^2}{144}.
$$

\medskip\noindent
{\bf Case 3:} Let $p_1\in (0,1)$. Since $|p_3|\le 1$, from \eqref{hankelp} it follows that
\begin{align*}
&|H_{2,1}(F_f/2)|\\&\le \cfrac{(2\lambda+1)^2p_1(1-p_1^2)}{96}\bigg(\bigg|\cfrac{(-4\lambda^2+4\lambda+11)p_1^3}{24(1-p_1^2)}+\cfrac{(2\lambda+5)p_1p_2}{6}-\cfrac{(2+p_1^2)p_2^2}{3p_1}\bigg|+1-|p_2|^2\bigg)\\
&=\cfrac{(2\lambda+1)^2 p_1(1-p_1^2)}{96}(|A+Bp_2+Cp_2^2|+1-|p_2|^2)
\end{align*}
where
$$
A=\cfrac{(-4\lambda^2+4\lambda+11)p_1^3}{24(1-p_1^2)}, B=\cfrac{(2\lambda+5)p_1}{6}, \text{ and } C=-\cfrac{(2+p_1^2)}{3p_1}.
$$
We note that $AC<0$. so we can apply case (ii) of Lemma \ref{Y(ABC)} and we obtain that only $|AB|-|C|(|B|-4|A|)\le 0$ is satisfying for some $p_1$, as we can see from below sub-cases.

{\bf 3(a)} Note that
$$
-4AC\biggl(\cfrac{1}{C^2}-1\biggr)-B^2=-\cfrac{2(-4\lambda^2+4\lambda+11)(4-p_1^2)+(2+p_1^2)(2\lambda+5^2)}{36(p_1^2+2)}\le 0.
$$
But the inequality $|B|<2(1-|C|)$ is equivalent to $(2\lambda+9)p_1^2-12p_1+8<0$ which is not true for $p_1\in (0,1)$ and $\lambda\in [1/2,1]$. 

{\bf3(b)} For $0<p_1<1$ and $1/2\le \lambda\le 1$, an easy computation shows that
$$
\min \biggl\{4(1+|C|)^2,-4AC\bigg(\cfrac{1}{C^2}-1\bigg)\biggr\}=-4AC\bigg(\cfrac{1}{C^2}-1\bigg)\le B^2.
$$
The last inequality follows from 3(a).

{\bf3(c)} We first show that $|C|(|B|+4|A|)-|AB|>0$ holds for all $0< p_1<1$  and $1/2\le \lambda\le 1$. A simple calculation shows that 
$$
|C|(|B|+4|A|)-|AB|=\cfrac{H(p_1)}{144(1-p_1^2)},
$$ 
where 
$$
H(x):=16(2\lambda+5)+8(17+6\lambda-8\lambda^2)x^2+(8\lambda^3-20\lambda^2-26\lambda-7)x^4.
$$
A slightly involved computation shows that  $H(x)>0$.
 
{\bf3(d)}
Next we compute
$$
|AB|-|C|(|B|-4|A|)=\cfrac{J(p_1)}{144(1-p_1^2)},
$$
where
$$
J(x):=(-8\lambda^3-44\lambda^2+90\lambda+183)x^4+8(-8\lambda^2+10\lambda+27)x^2-16(2\lambda+5).
$$
The discriminant of the quadratic equation $J(x)=0$ is given by
$$
\Delta=768 (137 + 113\lambda - 31 \lambda^2 - 24 \lambda^3 + 4 \lambda^4)\ge 0
$$
and it has following two solutions
$$
y_1:=\cfrac{-8(-8\lambda^2+10\lambda+27)-\sqrt{\Delta}}{2(-8\lambda^3-44\lambda^2+90\lambda+183)}\le 0
$$
and 
$$
y_2:=\cfrac{-8(-8\lambda^2+10\lambda+27)+\sqrt{\Delta}}{2(-8\lambda^3-44\lambda^2+90\lambda+183)}\ge 0.
$$
A simple computation shows that $y_2< 1$. Hence, it follows that
$$
\left \{
\begin{array}{ll}
	|AB|\le |C|(|B|-4|A|), & \quad \text{ for } 0<p_1\le \sqrt{y_2},
	\\[5mm]
|AB|\ge |C|(|B|-4|A|), & \quad \text{ for }\sqrt{y_2}\le p_1< 1.\\
\end{array}
\right.
$$ 
Therefore by Lemma \ref{Y(ABC)}, we obtain
\begin{equation*}
	|H_{2,1}(F_f/2)|\le \cfrac{(\lambda+1/2)^2p_1(1-p_1^2)}{24}(-|A|+|B|+|C|)=h(p_1),
\end{equation*}
where
$$
h(x):=\cfrac{(2\lambda+1)^2}{2304}\bigg({(4\lambda^2-12\lambda-39)x^4+4(2\lambda+3)x^2+16}\bigg).
$$
If $1/2\le \lambda\le 1$, we have $h'(s_3)=0$, where
\begin{equation}\label{eqp0}
s_3:=\sqrt{\cfrac{-2(2\lambda+3)}{4\lambda^2-12\lambda-39}}\in (0,\sqrt{y_2}]
\end{equation}
and we note that $s_3\le \sqrt{y_2}$. Since $h''(s_3)<0$, we have
\begin{equation}\label{eqh}
|H_{2,1}(F_f/2)|\le h(p_1)\le h(s_3)=\cfrac{(2\lambda+1)^2(12\lambda^2-60\lambda-165)}{576(4\lambda^2-12\lambda-39)}.
\end{equation}

{\bf3(e)} Next we consider $ \sqrt{y_2}\le p_1<1$ in order to complete the proof. Then, by Lemma \ref{Y(ABC)}, we have
	\begin{align*}
	&|H_{2,1}(F_f/2)|\\&\le \cfrac{(\lambda+1/2)^2p_1(1-p_1^2)}{24}	(|A|+|C|)\sqrt{1-\cfrac{B^2}{4AC}}\nonumber \\
	&=\cfrac{(2\lambda+1)^2(16-8p_1^2+(-4\lambda^2+4\lambda+3)p_1^4)}{2304}\sqrt{\cfrac{3((23+12\lambda-4\lambda^2)-(1+2\lambda)^2p_1^2)}{2(2+p_1^2)(-4\lambda^2+4\lambda+11)}}\\&=:T(p_1).
\end{align*}
By differentiating $T$, we obtain
\begin{align*}
T'(x)&=-\cfrac{(2\lambda+1)^2x}{2304(2 + x^2)^2 (-11 - 4 \lambda + 4 \lambda^2)}\sqrt{\cfrac{3(2+p_1^2)(-4\lambda^2+4\lambda+11)}{2((23+12\lambda-4\lambda^2)-(1+2\lambda)^2x^2)}}\\&\quad\times \bigg( 16 (-71 - 44\lambda + 4\lambda^2) + 
32 x^2 (13 + 35 \lambda - 7\lambda^2 - 16\lambda^3 + 4 \lambda^4)\\ &\quad+ 
x^4 (193 + 288\lambda - 344\lambda^2 - 192 \lambda^3 + 208\lambda^4)+4 x^6 (-3 + 2 \lambda) (1 + 2 \lambda)^3 \bigg)\\&<0.
\end{align*}
Therefore, $T$ is decreasing with respect to $x\in [\sqrt{y_2}, 1)$. Hence 
\begin{equation}\label{eqpsi}
T(p_1)\le T(\sqrt{y_2})=\cfrac{32(11+4\lambda-4\lambda^2)(2\lambda+1)^2(a+b\sqrt{c})}{2304~d^2}\sqrt{\cfrac{e-24(1+2\lambda)^2\sqrt{c}}{f+16(-4\lambda^2+4\lambda+11)\sqrt{c}}},
\end{equation}
where
\begin{align*}
a&:=2295+1740\lambda-504\lambda^2-336\lambda^3+48\lambda^4\\
b&:=-48-24\lambda+16\lambda^2\\
c&:= 3 (137 + 113\lambda - 31\lambda^2 - 24 \lambda^3 + 4 \lambda^4)\\
d&: = -8 \lambda^3 - 44 \lambda^2 + 90 \lambda + 183\\
e&:=3 (4317 + 4738\lambda - 104 \lambda^2 - 1040\lambda^3 - 48 \lambda^4 + 32 \lambda^5)\\
f&:=4 (-11 - 4\lambda + 4\lambda^2) (-129 - 70 \lambda + 28\lambda^2 + 8 \lambda^3).
\end{align*}
A tedious computations show that
$$
h(\sqrt{y_2})=T(\sqrt{y_2}),
$$
for each $\lambda\in [1/2,1]$. Therefore \eqref{eqpsi} together with \eqref{eqh} leads to 
$$
T(\sqrt{y_2})\le h(s_3)=\cfrac{(2\lambda+1)^2(12\lambda^2-60\lambda-165)}{576(4\lambda^2-12\lambda-39)}.
$$
Summarizing, from parts 1-3 it follows that
$$|H_{2,1}(F_f/2)|\le \cfrac{(2\lambda+1)^2(12\lambda^2-60\lambda-165)}{576(4\lambda^2-12\lambda-39)}.$$\\[2mm]

We now show that the above inequality is sharp by constructing extreme function. Consider the function $p_3$ of the form
$$
p_3(z)=\cfrac{1-z^2}{1-{2s_3}z+z^2}=1+2s_3z+(4s_3^2-2)z^2+(8s_3^3-6s_3)z^3+\cdots
$$
with $s_3$ given by \eqref{eqp0} and it belongs to the class $\mathcal{P}$ follows from Lemma \ref{Pclass}. The corresponding function $f_3$ can be obtain from \eqref{c(-1/2)p} and  
coefficients of $f_3$ are given by 
$$
a_2=\cfrac{(2\lambda+1)s_3}{2},\,a_3=\cfrac{(2\lambda+1)((3+2\lambda)s_3^2-1)}{6},
$$
 and,
 $$
 a_4=\cfrac{(2\lambda+1)(2\lambda+3)((2\lambda+5)s_3^2-3)s_3}{24}.
$$
From \eqref{Hankel}, it is clear that inequality is sharp for $f_3$. This completes the proof of the theorem.
\end{proof}

Choosing $\lambda=1/2$ and $\lambda=1$, we deduce the following sharp inequalities:
\begin{corollary}
If $f\in \mathcal{C}$, then
$$
|H_{2,1}(F_f/2)|\le 0.030303.
$$	
If $f\in \mathcal{C} (-1/2)$ given by \eqref{S}, then 
$$
|H_{2,1}(F_f/2)|\le 0.070811.
$$	
The inequalities are sharp.
\end{corollary}

\bigskip
\noindent
{\bf Acknowledgement.}
The first author thank SERB-CRG and the second author thank IIT Bhubaneswar for providing Institute Post Doctoral fellowship.

\end{document}